\documentclass[12pt]{amsart}
\usepackage{geometry} % see geometry.pdf on how to lay out the page. There's lots.
\usepackage{pgfplots}
\usepackage{psfrag}
\usepackage{amsmath}
\usepackage{calc}
\usepackage{systeme}
\usepackage{amsfonts}
\usepackage{amsthm}
\usepackage{algorithm2e}
\usepackage[applemac]{inputenc}
\usepackage{bbold}

\newtheorem{Theorem}{Theorem}[part]
\newtheorem{Lemma}{Lemma}[part]
\newtheorem{Rem}{Remark}
\newtheorem{Def}{Definition}
\newtheorem{Prop}{Proposition}[part]

\usepackage[foot]{amsaddr}

\title{Optimal stopping in mean field games, an obstacle problem approach}
\author{Charles Bertucci}
\thanks{This work is supported by a grant from the Fondation CFM pour la recherche.}
\address{Université Paris-Dauphine, PSL Research University,UMR 7534, CEREMADE, 75016 Paris, France}
\date{} % delete this line to display the current date

\begin{document}

\maketitle
\begin{abstract}
This paper is interested in the problem of optimal stopping in a mean field game context. The notion of mixed solution is introduced to solve the system of partial differential equations which models this kind of problem. This notion emphasizes the fact that Nash equilibria of the game are in mixed strategies. Existence and uniqueness of such solutions are proved under general assumptions for both stationary and evolutive problems.
\end{abstract}
\tableofcontents

\section{Introduction}
\subsection{General introduction}
This paper is the first of a series devoted to the systematic study of mean field games (MFG) of optimal stopping or impulse controls. We solve here a system of forward-backward obstacle problems which models MFG of optimal stopping without common noise. The case of the "Master equation" will be considered in a subsequent work.

In the past decade, MFG have been broadly studied since their introduction by Lasry and Lions in their series of paper [17,18,19] and it has been shown they have lots of applications [14,15]. From the well posedness of the problem [5,22] to the difficult question of the master equation [7,22] through numerical questions [1,3] and developments like long time average [10] or learning in MFG [9], the original model has been the source of a huge number of mathematical questions. A very powerful probabilistic point of vue has also been developed [11,16]. We refer to [22] for a large study of the problem. We briefly recall the model proposed by Lasry and Lions when there is no common noise. A continuum of indiscernable players, which is characterized by a measure $m$, faces an optimal control problem (which is stochastic here but can also be deterministic) whose value function is denoted by $u$. The players are initially described by the measure $m_0$ and are only interested in the repartition of the other players. Under general assumptions on the regularity and geometry of the cost functions (involved in both the running and terminal costs, respectively $f$ and $g$), the value function $u$ satisfies a Hamilton-Jacobi-Bellman equation given the evolution of the measure of the player. On the other hand, given the value function and then the behavior of the players, the measure $m$ satisfies a transport equation. Nash equilibria of the game are then given by solutions of the MFG system :

\begin{equation}
\begin{cases}
- \partial_t u - \Delta u + H(x, \nabla u) = f(x,m) \text{  in } \mathbb{R}^d \times(0,T)\\
\partial_t m - \Delta m - div(D_p(H(x, \nabla u))m) = 0 \text{  in } \mathbb{R}^d \times (0,T) \\
u(T) = g(T, m(T)) , m(0) = m_0
\end{cases}
\end{equation}

Our purpose is to study the analogue of $(1)$ for an optimal stopping problem. In this setting, players do not affect their velocity anymore but choose a time after which they definitively leave the game. We present our results in three different situations : a stationary case, a time dependent case with a general exit cost and a last case in which players can both leave the game and affect their velocity like in the case we just described. We are going to introduce the notion of mixed solutions for MFG, which describes Nash equilibria for MFG in mixed strategies. As we shall see, this notion is both analytically natural and the good notion of solution for MFG with optimal stopping. Recently some results have been obtained concerning this problem. In [23], the author solves a game of optimal stopping in a MFG and in [12], the authors prove an interesting and general result of existence in a probabilistic approach of the problem. We may also mention [13], in which the authors studied a MFG type system for an obstacle problem, which is the natural Hamilton-Jacobi-Bellman equation for an optimal stopping problem. Note that MFG with optimal stopping are very natural, from at least two perspectives. First, optimal stopping problems are interesting in themselves like for example resistance games or the modeling of american options. Then, it appears natural to allow the players to leave the MFG, because it is very restrictive to impose that all the players have to stay until the end like it is done in the classical MFG model. Indeed, it is more realistic to allow the players to leave the game by paying an exit cost.

\subsection{The model}
We present here the typical framework which is beneath the idea of optimal stopping in MFG. We refer to the next parts for more precise statements as we only want to give an intuition on why the model we are going to study is general and natural. We assume that there is an infinite number of players and we associate to each player $i$ a diffusion which satisfies :

\[
\begin{cases}
dX^i_t = \sqrt{2} dW^i_t\\
X^i_0 = x_i \text{ in } \Omega
\end{cases}
\]

here $\Omega$ is a bounded, smooth open set of $\mathbb{R}^d$ with $d \geq 1$. The player $i$ has to make the choice of a stopping time $\tau$ which has to be measurable for the $\sigma$-algebra generated by the $d$-dimensional brownian motion $(W^i_t)_{t \geq 0}$. We assume that all the brownian motions $((W^i_t)_{t \geq 0})_i$ are independent.  The cost is then defined by 

\[
\int_0^{\tau}f(X^i_s, m(s))ds + \psi(X_{\tau}^i, m({\tau}))
\]

where $m(t)$ is the measure which characterizes the repartition of the player in $\Omega$ at time $t$. The players minimize the expectation of this cost. We work with a finite horizon $T$ and assume that if the diffusion $(X^i_t)$ reaches the boundary $\partial \Omega$ of $\Omega$ at time $t^*$, then the player $i$ exits the game paying the cost $\int_0^{t^*}f(X^i_s, m(s))ds$. Given the evolution of the measure $m(t)$, the value function of the problem is a solution of the obstacle problem :

\begin{equation}
\begin{cases}
\max(-\partial_t u(t,x) -\Delta u(t,x) - f(x, m(t)), u(t,x) - \psi(x, m(t))) = 0\\ u(T,x) = \psi(x, m(T)) \\ \forall t \leq T, u(t,x) = 0 \text{ on } \partial \Omega
\end{cases}
\end{equation}

See for example [2] for more details on Hamilton-Jacobi-Bellman equations. On the other hand, given the value function $u$ and the fact that it is optimal to leave the game at $(t,x)$ if $u(t,x) = \psi(t,x)$, $m$ satisfies

\begin{equation}
\begin{cases}
\partial_t m -\Delta m = 0 \text{ on } \{(t,x) \in ]0,T[ \times \Omega / u(t,x) < \psi(x,m(t))\} \\
m(0) = m_0 \\
m = 0 \text{ elsewhere}
\end{cases}
\end{equation}

because the players move freely in the set $\{ u < \psi(m)\}$ since it is optimal to stay in the game. In this set, the players evolve only through the diffusion to which they are associated (they use no control). On the other hand, because it is optimal to leave the game in the set $\{u = \psi\}$, we want $m$ to be $0$ in this set, as all the players are leaving. We do not make precise here the sense in which systems $(2)$ and $(3)$ have to be taken as it will be the subject of later discussions. In the same way Nash equilibria in the classical MFG are given by solutions of the system (MFG), we expect Nash equilibria of our problem to be the solutions of 

\begin{equation*}{(OSMFG)}
\begin{cases}
\begin{aligned}
&\max(-\partial_t u(t,x) -\Delta u(t,x) - f(x, m(t)), u(t,x) - \psi(x, m(t))) = 0\\ & \forall (t,x) \in (0,T) \times \Omega \end{aligned}\\ 
\partial_t m -\Delta m = 0 \text{ on } \{(t,x) \in ]0,T[ \times \Omega / u(t,x) < \psi(x,m(t))\} \\
m = 0 \text{ elsewhere} \\u(T,x) = \psi(x, m(T)) \text{ on } \Omega \\ \forall t \leq T, u(t,x) = 0 \text{ on } \partial \Omega
\end{cases}
\end{equation*}

For pedagogical reasons, we first solve a stationary setting in which the time variable has disappeared and the leaving cost is $0$. In this setting, we are forced to introduce a source of players. This term can be interpreted as some players (who are still identical to all the others) entering the game uniformly in time. If we do not add this term, then the only equilibrium is $m=0$ because, almost surely, the trajectories of the players touch the boundary of $\Omega$ and thus the associated players leave the game. We denote it by $\rho$. This term is completely arbitrary and do not play a strong role in the qualitative approach we present. We also add a first order term in the equations in "$u$" and "$m$". The first one stands for a preference for the present over the future while the second one stands for a "natural" death rate of players. The "simpler" version of $(OSMFG)$ is then

\begin{equation*}{(SOSMFG)}
\begin{cases}
\max( -\Delta u + u - f(x, m), u) = 0 \text{ in } \Omega \\  -\Delta m + m = \rho \text{ on } \{u < 0\} \\
m = 0 \text{ on } \{ u = 0 \} \\ u = m = 0 \text{ on } \partial \Omega
\end{cases}
\end{equation*}

This system is the subject of the first part of this article and we extend our results to the case of $(OSMFG)$ in the second part. In the last part we present the case where both optimal stopping and continuous control are possible. This last setting leads to the following system :

\begin{equation*}{(COSMFG)}
\begin{cases}
\max(-\partial_t u - \Delta u + H(x, \nabla u) - f(m), u) = 0 \text{ in } (0,T) \times \Omega \\
\partial_t m - \Delta m - div(m D_p H(x, \nabla u)) = 0 \text{ in } \{ u < 0 \} \\
m = 0 \text{ in } \{ u = 0 \} \\
u = m = 0 \text{ on } \partial \Omega \\
m(0) = m_0 \text{ and } u(T) = 0 \text{ in } \Omega
\end{cases}
\end{equation*}

 where $H$ is the Fenchel conjugate of the part of the running cost which depends on the control. More details on this problem are given in the last part.\\
 \\
Please note that every time we consider the case $\psi(m) = 0$, we are in fact considering the case where $\psi$ does not depend on $m$ (it could depend on $x$), as we can change the cost $f$ with the addition of $\partial_t \psi + \Delta \psi$ to pass from a problem in which there is an exit cost to the case in which this cost is $0$.

\subsection{Assumptions}
We present here the regularity assumptions which hold for the rest of this discussion as well as some notations
\begin{itemize}
\item $\Omega$ is an open bounded subset of $\mathbb{R}^d$ with a smooth ( say $C^2$) boundary
\item $m_0 \in L^2(\Omega)$
\item $f$ is a contiuous application from $L^2(\Omega)$ to $L^2(\Omega)$
\item $\psi$ is a continuous application from $L^2((0,T), L^2(\Omega))$ to $L^2((0,T),H^1_0(\Omega) \cap H^2(\Omega)) \cap H^1((0,T), L^2(\Omega))$
\item $\rho \in C^{\infty}(\Omega)$ with compact support in $\Omega$ and $\rho \geq 0$.
\end{itemize}

\part{The stationary problem}
\section{Preliminary results}
In this section, we present useful results regarding the (stationary) problems (2) and (3). We address here the question of regularity of those problems in order to present the results of regularity for the MFG problem later on.
\subsection{The obstacle problem}
This problem is classical (see for instance [4]). For any $f$ of $L^2(\Omega)$, a solution $u$ of the obstacle problem is such that

\[
\max(-\Delta u + u - f, u) = 0 \text{ in } \Omega, \text{ } u \in H^2(\Omega)\cap H^1_0(\Omega)
\]

Existence and uniqueness of solutions of this problem can be found in [19]. Also, the mapping from $L^2(\Omega)$ to $H^2(\Omega)\cap H^1_0(\Omega)$ which associates to each $f$ the solution of the previous problem is continuous with respect to the canonical norms of those sets. Moreover, the sequence $(u_{\epsilon})_{\epsilon > 0}$ of $H^2(\Omega)\cap H^1_0(\Omega)$ converges to $u$, solution of the obstacle problem with source $f$, for the $H^1(\Omega)$ norm, if $u_{\epsilon}$ is defined for all $\epsilon$ as the unique solution of

\[
\begin{cases}
-\Delta u_{\epsilon} + u_{\epsilon} + \frac{1}{\epsilon} (u_{\epsilon})^+ = f \text{ in } \Omega \\
u_{\epsilon} = 0 \text{ on } \partial \Omega
\end{cases}
\]

\subsection{The equation in $m$}
We are here interested in the regularity of the solution $m$ of 

\[
\begin{cases}
-\Delta m + m = \rho \text{ in } \omega \\
m = 0 \text{ on } \partial \omega
\end{cases}
\]

where $\omega$ is any open subset included in $\Omega$. We extend $m$ by $0$ on $\Omega \setminus \omega$. Obviously, there exists a unique solution of this problem in $H^1_0(\omega)$ and we cannot expect further regularity without any assumptions on the regularity of $\omega$. ($H^1_0(\omega)$ is here defined as the closure of the $C^{\infty}$ functions with compact support in $\omega$ for the $H^1(\Omega)$ norm.)  Then, if $(\omega_n)_{n \geq 0}$ is a sequence of open subsets of $\Omega$ which converges towards $\omega$, open set of $\Omega$, with respect to the Hausdorff distance then the sequence of associated solutions $(m_n)_{n \geq 0}$ converges towards $m$ for the norm of $H^1(\Omega)$. Finally, if we define $m_{\epsilon}$ as the unique solution of

\[
\begin{cases}
-\Delta m_{\epsilon} + m_{\epsilon} + \frac{1}{\epsilon} \mathbb{1}_{\omega^c} m_{\epsilon} = \rho \text{ in } \Omega \\
m_{\epsilon} = 0 \text{ on } \partial \Omega
\end{cases}
\]

then $(m_{\epsilon})_{ \epsilon > 0}$ converges to $m$ for the norm of $L^2(\Omega)$. We also recall the following result, which is very simple and that we shall need later.

\begin{Lemma}
Let $m \in H^1_0(\Omega)$ be such that 

\begin{itemize}
\item there exists $\omega$ such that $ - \Delta m + m = \rho$  on  $\omega$ 
\item $ - \Delta m + m \leq \rho$ in the sense of distributions on $\Omega$
\end{itemize}

and denote by $u \in H^2(\Omega) \cap H^1_0(\Omega)$ a negative function, then 

\[
\int_{\Omega} (- \Delta u + u) m  \geq \int_{\Omega} u \rho
\]
and we have an equality if the support of $u$ is included in $\omega$

\end{Lemma}

\begin{proof}
This result is simply the translation in a variational form of the assumption $ - \Delta m + m\leq \rho$.

\end{proof}

\section{First properties of the system}
We now turn to the presentation of some results which give a better idea of the situation we are trying to model. Mainly we prove uniqueness and non existence of solutions for $(SOSMFG)$. Those results highlight that we have to adopt a new definition of solutions in order to find Nash equilibria of the MFG. We recall the notion of monotone applications on functions' spaces which is, as usual, crucial in the study of MFG. An application $T$ from $L^2(\Omega)$ into itself is said to be monotone if

\[
 \forall m_1, m_2 \in L^2(\Omega), \int_{\Omega}(T(m_1) - T(m_2))(m_1 - m_2) \geq 0
\]

It is strictly monotone if the inequality is strict as soon as $m_1 \ne m_2$. $T$ is said to be anti-monotone if $-T$ is monotone.\\
\\
We first prove a uniqueness result which depends on the monotonicity of the cost $f$.

\begin{Theorem}
If $f$ is strictly monotone, then there exists at most one solution of $(SOSMFG)$.
\end{Theorem}
\begin{proof}
We shall present here an analogue of the proof of uniqueness of Lasry and Lions in [19]. We denote by $(u_1, m_1)$ and $(u_2, m_2)$ two solutions of $(SOSMFG)$ and note $ u = u_1 - u_2$, $m = m_1 - m_2$ and the continuation sets $\Omega_1 = \{u_1 < 0 \}$ and $\Omega_2 = \{u_2 < 0 \}$. We then compute :

\begin{equation}
\begin{aligned}
\int_{\Omega_1 \cup \Omega_2} (-\Delta u + u) m = &  \int_{\Omega_1 \cap \Omega_2} (f(m_1) - f(m_2))m \\ &+  \int_{\Omega_1 \setminus \Omega_2} f(m_1)m_1 + \int_{\Omega_2 \setminus \Omega_1} f(m_2)m_2 
\end{aligned}
\end{equation}

We here use the fact that on $\Omega_i^c$, $-\Delta u_i = 0$ which holds because the obstacle problem 

\[
\max(-\Delta u_i + u_i -f(m_i),u_i) = 0 \text{ in } \Omega
\]

holds true in $L^2(\Omega)$.\\
\\
Let us remark that on $\Omega_2^c$, $m_2 = 0$ and $f(m_2) \geq 0$ because $(u_2,m_2)$ is a solution of $(SOSMFG)$. Thus, it is in particular true on $\Omega_1 \setminus \Omega_2$. Hence the following inequality is true.

\[
\int_{\Omega_1 \setminus \Omega_2} f(m_1)m_1\geq \int_{\Omega_1 \setminus \Omega_2} (f(m_1) - f(m_2))(m_1 - m_2)
\]

With the same argument on the third term of the right hand side of $(4)$, we obtain 

\[
 \int_{\Omega_1 \cup \Omega_2} (-\Delta u + u) m  \geq  \int_{\Omega_1 \cup \Omega_2} (f(m_1) - f(m_2)) m \geq 0
\]

We now evaluate the sign of this term in a different way. First, let us remark that because of the fact that $(u_1, m_1)$ and $(u_2, m_2)$ are solutions of $(SOSMFG)$, we know using lemma 1.1 that

\[
\int_{\Omega} (-\Delta u_1 + u_1) m_1  +  \int_{\Omega} (-\Delta u_2 + u_2) m_2  = \int_{\Omega} (u_1 + u_2) \rho
\]

This equality gives 

\[
\begin{aligned}
\int_{\Omega} (-\Delta u + u) m  &= - \int_{\Omega} (-\Delta u_1 + u_1) m_2 + \int_{\Omega} u_1 \rho \\
& - \int_{\Omega} (-\Delta u_2 + u_2) m_1 + \int_{\Omega} u_2 \rho
\end{aligned}
\]

Using once again lemma 1.1 we deduce

\[
\begin{aligned}
\int_{\Omega} (-\Delta u + u ) m \leq 0
\end{aligned}
\]

And finally we obtain
\[
\int_{\Omega} (-\Delta u + u) m = \int_{\Omega_1 \cup \Omega_2} (f(m_1)- f(m_2))m = 0
\]

So by strict monotonicity, $m_1 = m_2$ and the result is proved.

\end{proof}

We observe that the monotonicity of $f$ is important for the question of uniqueness.We now show it is in general necessary.

\begin{Prop}
Uniqueness does not hold in general for $(SOSMFG)$.
\end{Prop}
\begin{proof}
Once again we denote by $m^* \in H^1_0(\Omega)$ the unique solution of 

\[
-\Delta m^* + m^* = \rho \text{ in } \Omega
\]

We note $E(m) = \int_{\Omega}|x|m(x)dx$. We define $f(m)$ by the following equation for $m \in L^2(\Omega)$ :
\[
\forall x \in \Omega, f(x,m) = -2 \frac{E(m)}{E( m^*)} + 1
\]

We observe that $f(m^*) = -1 $ and $f(0) = 1$. It is clear that (0,0) is a solution of $(SOSMFG)$. Now if we denote by $u^*$ the solution of $(1)$ with $m = m^*$, it is clear by the maximum principle that the contact zone with the obstacle $\{u = 0 \}$ is empty and thus that $(u^*, m^*)$ is also a solution of $(SOSMFG)$.
\end{proof}

We take advantage of this collection of remarks around this new system to present the fact that uniqueness of solutions does not hold in general, even if $f$ is strictly monotone. However, we are not interested in giving too much detail about the case in which the obstacle $\psi$ depends on $m$, as it is the subject of the next part.

\begin{Prop}
For all strictly monotone $f$, there exists an obstacle $\psi$ such that there is no uniqueness of solutions. 
\end{Prop}

\begin{proof}
Once again we denote by $m^* \in H^1_0(\Omega)$ the unique solution of 
\[
-\Delta m^* + m^* = \rho  \text{ in } \Omega
\]

which is strictly positive by the strong maximum principle. We define $u^* \in H^1_0(\Omega)\cap H^2(\Omega)$ by the unique solution of 
\[
-\Delta u^* + u^* = f(m^*) \text{ in } \Omega
\]

and $u_* \in H^1_0(\Omega)\cap H^2(\Omega)$ by the unique solution of 

\[
-\Delta u_* + u_* = f(0)
\]

Now define $\psi$ by the following :

\[
\psi(m) = (u^*)\frac{m}{m^*} + \frac{m^* - m}{m^*}(u_*)
\]

It is then easy to verify that both $(u^*, m^*)$ and $(u_* , 0)$ are solutions of $(SOSMFG)$.

\end{proof}
\begin{Rem}
We can notice that here $\psi$ is monotone. We shall see later that in order to have uniqueness of solutions, we have to make an assumption which is somehow related to the anti-monotonicity of $\psi$.
\end{Rem}

We now present an example of non existence of solutions for $(SOSMFG)$. It is very general in its construction and the reader could easily understand how it could be adapted for different models.

\begin{Prop}
There exists $f$ such that there is no solution for the system $(SOSMFG)$.
\end{Prop}
\begin{proof}
We define by $m^*$ the unique solution of the equation  :

\[
\begin{cases}
- \Delta m + m = \rho \text{ in } \Omega \\
m = 0 \text{ on } \partial \Omega
\end{cases}
\]

We then choose a negative smooth function ($C^2$) denoted by $u^*$ which vanishes only on  $\{ x_0 \}$ and on the boundary of the domain, where $x_0 \in \Omega$. We then define $f$ on $L^2(\Omega)$ by

\[
f(x,m) =  - \Delta u(x) + u(x) + m(x) - m^*(x), \forall x \in \Omega
\]

By construction

\[
\begin{cases}
- \Delta u^* + u^* = f(m^*), \forall x \in \Omega\\ 
u \leq 0
\end{cases}
\]

Hence, $u^*$ satisfies the  obstacle problem with source $f(m^*)$. But since $m^*$ is strictly positive on $\Omega$ by the strong maximum principle, $m^*$ does not satisfy $m^* = 0 $ on $\{ u^* = 0 \}$. Thus the couple $(u^*,m^*)$ is not a solution of $(SOSMFG)$. Suppose that there exists a solution $(u,m)$ of $(OSMFG)$, then necessarily $m < m^*$ (i.e. $m\leq m^*$ and $m \ne m^*$). Then recalling the strict monotonicity of $f$ and the strong maximum principle we obtain that $u < u^*$ everywhere on the domain. So the contact zone $\{ u = 0 \}$ is empty. This contradicts the fact that $m < m^*$, thus we have proven that there is no solution to $(SOSMFG)$ in this case. 

\end{proof}

\begin{Rem}
One can object to this conclusion that imposing that $m$ vanishes (in some sense) on $\{u^* = 0 \}$ is  very restrictive, regarding the problem we are trying to model, because this set has a $0$ Lebesgue measure. One can think that we should not have any constraint in this case. The interested reader could easily note that if $u^*$ is such that it vanishes only on a small ball around $x_0$, then we could also have proven non existence in this case.
\end{Rem}

We have now proven that a solution of the system may not exist if $f$ is not anti-monotone and that it may not be unique if $f$ is not monotone. This observation leads us to consider a relaxed notion of solutions of $(SOSMFG)$. 

\section{Towards the good notion of solutions}
We now present the notion of mixed solutions of $(SOSMFG)$ for which we can establish better results of existence than for the notion of solution we used in the previous section. We give two different approaches for $(SOSMFG)$, each of them leads to this notion of mixed solution. The first one is an optimal control interpretation of $(SOSMFG)$. The second one is a penalized version of the problem.

\subsection{The optimal control interpretation}
We here assume that the dependence of $f$ in $m$ is local, meaning that
\[
\forall m \in L^2(\Omega), x \in \Omega, f(x,m) = f(x,m(x))
\]

We then suppose that there exists $\mathcal{F}$ such that :

\[
\forall (x,p) \in \Omega \times \mathbb{R}, \frac{\partial \mathcal{F}}{\partial p}(x,p) = f(x,p)
\]
We assume that $\mathcal{F}(x, \cdot) \in C^1(\mathbb{R})$ and that $\mathcal{F}$ is continuous as seen as a functional from $H^1_0(\Omega)$ into $L^1(\Omega)$ (which is of course compatible with the fact that $f \in C^0(L^2(\Omega))$). 

We now introduce an optimal control problem using $\mathcal{F}$ which leads to $(SOSMFG)$. We start by recalling the optimal control interpretation for the classical MFG system (1). This approach was presented in [19] and has been well studied since (see [3,6,8] for example). Each time the authors of those articles link the system (1) with both an optimal control of a Fokker-Planck equation and an optimal control of a Hamilton-Jacobi-Bellman equation. Here we only present (formally) the link with the Fokker-Planck equation. For each control $\alpha$, we define $m$ as the solution of the Fokker-Planck equation :

\[
\begin{cases}
\partial_t m - \Delta m - div( \alpha m) = 0 \text{ in } ]0,T[ \times \Omega \\
m(0) = m_0 \text{ in } \Omega
\end{cases}
\]

and we introduce the following problem 

\[
\inf_{\alpha} \big\{ \int_0^T \int_{\Omega} \mathcal{F}(m)  + L(x, \alpha)m \big\}
\]

Where $L$ is the Fenchel conjugate of the Hamiltonian $H$ with respect to the $p$ variable. Then it has been proved ([6] for example) that the infimum is reached when $\alpha(t,x) = -D_pH(x, \nabla u(t,x))$, where $(u,m)$ is the solution of $(1)$. Our aim is now to prove a similar result for $(SOSMFG)$. We start by a remark : in the optimal control problem for the players, the term $-D_pH(x, \nabla u(t,x))$ is an optimal choice to minimize their cost (it is their best answer), if $(u,m)$ is a solution of $(1)$. Thus, the problem of controlling the Fokker-Planck equation is interpreted as prescribing a behavior for the players (it is the $\alpha$, the control) and then minimizing a quantity depending on the repartition of the players ($m$) induced by the prescribed behavior ($\alpha$). For the optimal stopping case, we can describe a behavior as an exit set in $\Omega$, on which the players leave the game. Thus, we can search for an optimal control problem where the control $\alpha$ is now a set and $m$ (the associated repartition of players) satisfies the equation associated with this behavior : 
\[
\begin{cases}
 - \Delta m + m  = \rho \text{ in } {\alpha} \\
m = 0 \text{ on } \partial \alpha
\end{cases}
\]

We then minimize the following 
\[
\inf_{\alpha}  \big\{ \int_{\Omega} \mathcal{F}(m)   \big\}
\]

where the infimum is taken over the open sets of $\Omega$. As it is well known minimization problems over a collection of open sets are difficult in general. This is why, we shall work on a relaxed version of the problem. We make two remarks : for all $\alpha$, $m$ is positive by the maximum principle and the following inequality is satisfied in the sense of distributions 
\[
- \Delta m + m \leq \rho
\]

So a possible relaxation for the previous problem is
 
 \begin{equation}
 \inf_{m \in \mathcal{H}}  \big\{\int_{\Omega} \mathcal{F}(m)   \big\}
 \end{equation}

where $\mathcal{H} = \{ m \in H^1_0(\Omega), m \geq 0,  - \Delta m + m \leq \rho \}$ in which the second inequality holds in the sense of distributions.

\begin{Theorem}
Assume $\mathcal{F}$ is strictly convex, then the previous relaxed problem admits a unique solution. Moreover, any minimizer $m$ of the problem satisfies the Euler-Lagrange optimal conditions :

\[
\forall m' \in \mathcal{H},  \int_{\Omega} f(m)(m' - m) \geq 0
\]
\end{Theorem}

\begin{proof}
The functional $\mathcal{F}$ is weakly sequentially lower semi continuous (ws lsc) for the topology of $H^1_0(\Omega)$ as it is both continuous and convex. Take a minimizing sequence $(m_n)_{n \in \mathbb{N}}$, by definition,
 
 \[
 - \Delta m_n  + m_n \leq \rho
 \]
 
 hence testing this relation against $m_n$ itself,
 \[
 \int_{\Omega} |\nabla m_n|^2 + \int_{\Omega} m_n^2 \leq \int_{\Omega}\rho m_n \leq ||\rho||_{L^2(\Omega)} ||m_n||_{L^2(\Omega)}
 \]
 
 which implies a bound in $H^1_0(\Omega)$. So there is a subsequence of $(m_n)_{n \in \mathbb{N}}$ which converges toward a limit $m$ in $H^1_0(\Omega)$ for the weak topology. It is easy to check that the limit still belongs to $\mathcal{H}$. Then because $\mathcal{F}$ is ws lsc, we get that $m$ is a minimizer of the problem. And it is obviously unique if $\mathcal{F}$ is strictly convex. Because of the regularity of $\mathcal{F}$, the verification of the Euler-Lagrange inequality is trivial.
\end{proof}

From now on we will denote by $m$ the unique minimizer of this problem and $u$ the solution of 

\[
\begin{cases}
\max(-\Delta u + u - f(m), u) = 0 \text{ in } \Omega \\
u = 0 \text{ on } \partial \Omega
\end{cases}
\]

Our goal is to show that the couple $(u, m)$ is a solution of $(SOSMFG)$ in a certain sense. More precisely we obtain the following :

\begin{Theorem}
$(u, m)$ satisfies :

\[
\begin{cases}
\max(-\Delta u + u - f(m), u) = 0 \text{ in } \Omega \\
 - \Delta m + m=  \rho \text{ on } \{u < 0 \} \\
m = 0 \text{ on } \partial \Omega \\
\int_{ \Omega} f(m)m = \int_{\Omega} u \rho
\end{cases}
\]
\end{Theorem}

\begin{proof}
For all $m'$ in $\mathcal{H}$, we compute

\[
\begin{aligned}
\int_{\Omega} f(m)(m' - m) & =  \int_{\{ u < 0\}} f(m)(m' - m) + \int_{ \{ u = 0 \}} f(m)(m' - m)\\
& =  \int_{\{ u < 0\}} ( - \Delta u + u)(m' - m) +  \int_{ \{ u = 0 \}} f(m)(m' - m)\\
& =  \int_{\Omega} ( - \Delta u + u)(m' - m) + \int_{ \{ u = 0 \}} f(m)(m' - m)
\end{aligned}
\]

Where the last equality holds true because $u  \in H^2(\Omega)$. Take the sequence $(m_1^{\epsilon})_{\epsilon > 0}$ defined as solutions of 

\[
\begin{cases}
- \Delta m_1^{\epsilon} + m_1^{\epsilon} + \frac{1}{\epsilon} \mathbb{1}_{\{ u = 0 \}} (m_1^{\epsilon} - m)^+ = \rho \text{ in } \Omega\\
m_1^{\epsilon}  = 0 \text{ on } \partial \Omega
\end{cases}
\]

For all $\epsilon > 0$, $m_1^{\epsilon} \in H^2(\Omega) \cap H^1_0(\Omega)$. Next, observe that in the sense of distributions

\[
-\Delta m_1^{\epsilon} + m_1^{\epsilon} \leq \rho
\]

\[
m_1^{\epsilon} \geq 0 \text{ in } \Omega
\] 

Hence, extracting a subsequence  if necessary, $(m_1^{\epsilon})_{\epsilon}$ converges weakly to a limit $m_1$ in $\mathcal{H}$ such that $m_1 \leq m$ on $\{ u = 0 \}$ and

\[
\int_{\Omega} (-\Delta u + u)m_1 = \int_{\Omega} u \rho
\]

Thus we can write 

\[
 \int_{\Omega} f(m)(m_1 - m) = \int_{\Omega} u \rho + \int_{\Omega} ( - \Delta u + u)(- m) +  \int_{ \{ u = 0 \}} f(m)(m_1 - m)
\]

Now, note that because $u$ satisfies an obstacle problem 

\[
f(m) \geq 0 = -\Delta u + u \text{ on } \{u = 0 \}
\]

Using this remark and lemma 1.1 we deduce

\[
\int_{\Omega} f(m)(m_1 - m) \leq 0
\]

And the inequality is strict if 
\[
-\Delta m + m \ne \rho \text{  on } \{u < 0 \}
\]

Thus, because of the Euler-Lagrange conditions of optimality it follows that
\begin{equation*}
- \Delta m + m= \rho \text{ on } \{u < 0\}
\end{equation*}

Now we define a sequence $(m_2^{\epsilon})_{\epsilon> 0}$ as the solutions of

\[
\begin{cases}
 - \Delta m_2^{\epsilon} + m_2^{\epsilon} + \frac{1}{\epsilon} \mathbb{1}_{\{u = 0 \}} m_2 = \rho \text{ in } \Omega \\
m_2^{\epsilon} = 0 \text{ on } \partial \Omega
\end{cases}
\]

Once again, extracting a subsequence if necessary, $(m_2^{\epsilon})_{\epsilon}$ converges weakly to a limit $m_2 \in \mathcal{H}$ and we deduce  

\[
\int_{\Omega} f(m)(m_2 - m) =  \int_{\{u = 0\}} f(m)(m_2 - m) = -\int_{\Omega}f(m) m
\]

because both $m$ and $m_2$ satisfy 

\[
-\Delta m' + m' = \rho \text{ in } \{u<0 \}
\]

Remark that because of the Euler-Lagrange conditions, we obtain that

\[
\int_{\{u = 0\}} f(m)m \leq 0
\]

But because of the obstacle problem satisfied by $u$, this quantity is also positive. Hence 
\begin{equation}
\int_{\{u = 0\}} f(m)m = 0
\end{equation}

Now we only have to remark that because $(6)$ holds true,

\[
\begin{aligned}
\int_{ \Omega} f(m)m & = \int_{\{u < 0 \}} f(m)m\\
& = \int_{\{u < 0 \}} (-\Delta u + u)m\\
& = \int_{\Omega} u \rho
\end{aligned}
\]

\end{proof}

We now introduce the following definition :

\begin{Def}
A pair $(u,m) \in (H^1_0(\Omega)\cap H^2(\Omega)) \times H^1_0(\Omega)$ is a mixed solution of $(SOSMFG)$ if :
\begin{itemize} 
\item  $\max(-\Delta u + u - f(m), u) = 0$ in $\Omega$
\item  $ - \Delta m + m = \rho$ on $\{u < 0 \}$ in $\mathcal{D}'(\Omega)$
\item $ -\Delta m + m \leq \rho$ in $\mathcal{D}'(\Omega)$
\item $\int_{\{ u =0 \}} f(m)m = 0$
\end{itemize}
\end{Def}

Obviously if $m$ is a minimizer of $(5)$ and $u$ is the solution of the obstacle problem with source $f(m)$, then $(u,m)$ is a mixed solution of $(SOSMFG)$. We present here in which extend it is natural to adopt this definition. This discussion is quite formal and we begin by taking a mixed solution $(u,m)$ of $(SOSMFG)$. Let us remark that formally there exits a potential $V$ such that 

\[
- \Delta m + m + Vm = \rho \text{ in } \Omega
\]

with the convention that $V$ equals $\infty$ when $m$ equals $0$. Because $ - \Delta m + m \leq \rho$, $V \geq 0$. The classical interpretation for such a term in mathematical modeling is usually a death rate. It is more appropriate to talk about a leaving rate here as we are interested in an optimal stopping game. We can now describe a Nash equilibrium of the MFG using this potential $V$ : in the zone $\{ V = 0 \}$ it is optimal for the players to stay in the game ; in the zone $\{V = \infty\}$ it is optimal for the players to leave the game ; the zone $\{ 0 < V < \infty \}$ describes an indifference region where it is both optimal to leave and to stay and where players do leave with a non constant leaving rate given by $V$. We can remark we can still have a Nash equilibrium even if the players behave differently in the same situation, as soon as the way with which they choose between the options is random and that the law of this choice is the same for all the players. A Nash equilibrium of the MFG is then obtained when the leaving rate $V$ of the players in this region is such that $0= - \Delta u + u = f(m)$. Because the first equality comes from the fact that it is optimal to leave, and the second one from the fact that it is optimal to stay.\\
It justifies our choice of the terminology mixed solution, as it corresponds to the players playing in mixed strategies. Remark that here a strategy for the players is not just an exit set $\{ u = 0\}$ but also a leaving rate $V$ on this set which characterizes the probability they are playing as a strategy.\\
\\
From a game theory point of vue it is very common to have existence of equilibria in mixed strategies but not in pure strategies. But from a MFG point of vue, it is quite a surprise that Nash equilibria for the game cannot be found in pure strategies for the optimal stopping problem. We recall that in the classical MFG setting, under the monotonicity assumptions on the costs $f$ and $g$, there exists a unique solution for the system $(1)$, which is interpreted as a Nash equilibrium in pure strategies.\\
\\
The notion of mixed solutions of MFG we just presented is highly related to the notion of weak solutions which is introduced in [12]. In [12], the measure of stopping time is a weak solution, when it satisfies a fixed point property in a weak sense. It can be seen as letting the players play in mixed strategies. The authors of this article oppose their notion of weak solution to a notion of strong solution, which is more or less a classical notion of solution for $(SOSMFG)$. They prove their existence under a monotonicity assumption on the cost. We present the same result with a partial differential equations point of vue. We begin with a lemma.

\begin{Lemma}
Suppose that $f$ is anti-monotone and define $T : H^1_0(\Omega) \to H^1_0(\Omega)$ by the fact that $T(m)$ is the only solution of 

\[
-\Delta T(m) + T(m) = \rho \text{ in } \{ u < 0 \}
\]

where $u$ is the only solution of 

\[
\max(-\Delta u + u - f(m), u) = 0 \text{ in } \Omega
\]

Then $T(m_1) \leq T(m_2)$ if $m_1 \leq m_2$.
\end{Lemma}
\begin{proof}
Take $m_1 \leq m_2$ and define $u_1$ and $u_2$ by 

\[
\max(-\Delta u_i + u_i - f(m_i), u_i) = 0 \text{ in } \Omega
\]

for $i = 1, 2$.  Since $u_1 \geq u_2$ ($f$ is anti-monotone) we deduce that $\{u_1 < 0 \} \subset \{u_2 < 0 \}$. Thus $m$ defined by $m = T(m_2) - T(m_1)$ satisfies 
\[
- \Delta m + m = 0 \text{ on } \{u_1 < 0 \}
\]

$T(m_2)$ is positive  on $\{u_1 < 0 \} \subset \{u_2 < 0 \}$ by the maximum principle. Since $T(m_1)$ vanishes everywhere on $\partial \{u_1 < 0 \}$ we conclude by the maximum principle that $T(m_1) \leq T(m_2)$.
\end{proof}
We can now prove the following :

\begin{Theorem}
Suppose that $f$ is anti-monotone, then there exists a smallest "classical" solution $(u^*, m^*)$ of $(SOSMFG)$ such that all equations are satisfied in $L^2(\Omega)$. It is the smallest solution in the following sense :  if $(u,m)$ is a solution of $(SOSMFG)$, then $m^* \leq m$.
\end{Theorem}

\begin{proof}
We set $m_0 = 0$. For all $n \in \mathbb{N}$, we define :
\begin{itemize}
\item $m_{n+1} = T(m_n)$ 
\item $u_n$ the solution of the obstacle problem with source $f(m_{n})$
\end{itemize}

By the maximum principle, $m_1 \geq 0 = m_0$. Using the previous lemma we get by induction that $(m_n)_{n \in \mathbb{N}}$ is an increasing sequence, while $(u_n)_{n \in \mathbb{N}}$ is a decreasing one. Recalling the estimates on the obstacle problem ([4]) , the sequence $(u_n)_{n \in \mathbb{N}}$ converges pointwise to a limit we call $u^* \in H^2(\Omega) \cap H^1_0(\Omega)$. Since $(u_n)_{n \in \mathbb{N}}$ is decreasing, the sequence of open sets $\{u_n < 0 \}$ converges to $\{u^* < 0\}$ for the Hausdorff distance. Then recalling the result we gave in the previous section, we deduce that $(m_n)_{n \in \mathbb{N}}$ converges to $m^*$ which satisfies

\[
-\Delta m^* + m^* = \rho \text{ in } \{ u^* < 0 \}
\]
 and is equal to $0$ elsewhere. Thus, the couple $(u^*,m^*)$ is a solution of the system $(SOSMFG)$. It is then easy to prove that it is the smallest solution of the system. Indeed, if we take another solution $(u,m)$ of the system then obviously $m \geq m_0 = 0$. Since $m$ is a fixed point for the application $T$, using the previous lemma, for all $n \in \mathbb{N}$, $m \geq m_n$, which proves the last point of the theorem by passing to the limit.
\end{proof}

\subsection{The penalized system}
We now present a penalized version of $(SOSMFG)$ which leads to the existence of mixed solutions. A natural penalization for $(SOSMFG)$ is the coupling of the penalized version of both the obstacle problem and the Laplace's equation on the domain $\{u < 0 \}$ . This leads to the following penalized system of partial differential equations : 

\[
\begin{cases}
 - \Delta u + u + \frac{1}{\epsilon} (u)^+ = f(m) \text{ in } \Omega \\
 - \Delta m + m + \frac{1}{\epsilon} \mathbb{1}_{\{u \geq 0\}}m = \rho \text{ in } \Omega \\
 u = m = 0 \text{ on } \partial \Omega
\end{cases}
\]

The lack of continuity of this system prevents us from proving existence of solutions. Indeed the equation in $m$ has no continuity with respect to $u$ and this problem cannot be overcome. The reason why is basically the same as the reason why there is no existence of classical solutions for $(SOSMFG)$. The proof of non-existence of solutions for $(SOSMFG)$ can be easily adapted to prove its counterpart for this penalized system so we shall not present it here once again. We add a new unknown $\alpha$ to introduce the leaving rate $V$ in the penalized system. This approach makes the problem more convex and allows us to prove existence of solutions. We say that $(u, m, \alpha)$ is a solution of the penalized system if 

\begin{equation}
\begin{cases}
- \Delta u + u + \frac{1}{\epsilon} (u)^+ = f(m) \text{ in } \Omega \\
- \Delta m + m + \frac{1}{\epsilon}\alpha \mathbb{1}_{\{u \geq 0\}}m = \rho \text{ in } \Omega \\
u = m = 0 \text{ on } \partial \Omega \\
0 \leq \alpha \leq 1 ; u \ne 0 \Rightarrow \alpha = 1
\end{cases}
\end{equation}

We are now able to prove the following :

\begin{Theorem}
For all $\epsilon > 0$ there exists a solution $(u_{\epsilon}, m_{\epsilon}, \alpha_{\epsilon})$ of $(7)$. 
\end{Theorem}
\begin{proof}
We define first the application $\mathcal{F}_1$ from $L^2(\Omega)$ into $H^1_0(\Omega) \cap H^2(\Omega)$ by : $\mathcal{F}_1(m)$ is the only solution of the obstacle problem with source $f(m)$. We also define the correspondance $\mathcal{F}_2$ from $ H^1_0(\Omega) \cap H^2(\Omega)$ into $L^2(\Omega)$ by : 

\[
\mathcal{F}_2(u) = \{ m \in H^1_0(\Omega) \cap H^2(\Omega), \exists \alpha \in \mathcal{D}(u), (u,m,\alpha) \text{ solve } (3^*) \}
\]

where $\mathcal{D}(u) = \{ \alpha \in L^{\infty}(\Omega) , 0 \leq \alpha \leq 1, u \ne 0 \Rightarrow \alpha = 1 \}$ and $(3^*)$ is the following 

\begin{equation*}{(3^*)}
\begin{cases}
 - \Delta m + m + \frac{1}{\epsilon}\alpha \mathbb{1}_{\{u \geq 0 \}}m = \rho \\
m = 0 \text{ in } \partial \Omega
\end{cases}
\end{equation*}

Thus finding a solution $(u,m, \alpha)$ of the penalized system is equivalent to finding $m \in \mathcal{F}_2(\mathcal{F}_1(m) )$. Recalling the results of section 3 the application $\mathcal{F}_1 $ is well defined and continuous. As we are going to apply Kakutani's fixed point theorem, we just have to verify that the correspondance $\mathcal{F}_2$ is upper semicontinuous (i.e. that for all open set $\mathcal{O} \subset H^1_0(\Omega)$, $\{ u \in  H^1_0(\Omega) \cap H^2(\Omega), \mathcal{F}_2(u) \subset \mathcal{O} \}$ is open) and takes values in the set of convex closed subsets of $H^1_0(\Omega)$. As the last point is trivial we focus on the upper semicontinuity. 
We take an open set $\mathcal{O} \subset H^1_0(\Omega)$ and $u \in H^1_0(\Omega) \cap H^2(\Omega)$ such that $\mathcal{F}_2(u) \subset \mathcal{O}$.\\
\\
We are now going to find a small enough $\delta > 0$ such that for any $v \in H^1_0(\Omega) \cap H^2(\Omega)$ such that $||u - v ||_{H^2(\Omega)} \leq \delta$,  $\mathcal{F}_2(v) \subset \mathcal{O}$. First we remark that 

\[
dist(\mathcal{F}_2(u) , \mathcal{O}^c) > 0
\]

Indeed, $\mathcal{F}_2(u)$ is a compact subset of $\mathcal{O}$ because it is a bounded subset of $H^1_0(\Omega) \cap H^2(\Omega)$. We call $a = dist(\mathcal{F}_2(u) , \mathcal{O}^c) > 0$. We now prove that for every $m' \in \mathcal{F}_2(v)$, there exists $m \in \mathcal{F}_2(u)$ such that $|| m - m'||_{L^2} \leq \frac{a}{2}$, given that $\delta$ is small enough. Hence, $\mathcal{F}_2(v) \subset \mathcal{O}$ shall hold.\\
\\
Take $m' \in \mathcal{F}_2(v)$, there exists $\alpha'$, taking values between $0$ and $1$, such that 

\[
\begin{cases}
- \Delta m' + m' + \frac{1}{\epsilon}\alpha' \mathbb{1}_{\{v \geq 0 \}}m' = \rho \text{ in } \Omega \\
m' = 0 \text{ on } \partial \Omega
\end{cases}
\]

We divide $ \Omega$ into three zones and define $\alpha$ by:
\begin{itemize}
\item on $\{v \geq 0 \} \cap \{ u = 0 \}$, $\alpha = \alpha '$
\item on $\{v < 0 \} \cap \{ u = 0 \}$, $\alpha = 0$
\item on the rest of $\Omega$,  $\alpha = 1$
\end{itemize}

Then we define $m$ as the solution of

\[
\begin{cases}
 - \Delta m + m + \frac{1}{\epsilon}\alpha \mathbb{1}_{\{v \geq 0 \}}m = \rho \\
m = 0 \text{ on } \partial \Omega
\end{cases}
\]

and we set $\mu = m' - m$. $\mu$ solves

\[
\begin{cases}
\begin{aligned}
- \Delta \mu + \mu + \frac{1}{\epsilon}\alpha' \mathbb{1}_{(\{v \geq 0 \} \cap \{ u \geq 0 \})}\mu = & -\frac{1}{\epsilon}( \mathbb{1}_{(\{v > 0 \} \cap \{ u < 0 \})}m' - \mathbb{1}_{(\{v > 0 \} \cap \{ u < 0 \})}m\\
 &+ \mathbb{1}_{(\{v = 0 \} \cap \{ u < 0 \})}m' + \mathbb{1}_{(\{v = 0 \} \cap \{ u > 0 \})}(\alpha' m' - m ))\\
 & \text{ in } \Omega
 \end{aligned} \\
\mu = 0 \text{ on } \partial \Omega
\end{cases}
\]

All the terms of the right hand side involve characteristic functions of subsets where $u$ and $v$ do not have strictly the same sign. We claim that we can always choose $\delta$ sufficiently small to make the second term has small as we want for the norm of $L^2(\Omega)$, independently of $m$ and $m'$. Hence taking $\delta $ small enough, we obtain that $|| m' - m ||_{L^2} \leq \frac{a}{2}$. Thus we can apply Kakutani's fixed point theorem and find a solution of the penalized system.

\end{proof}
We could have used a smoother version of the penalized system to prove existence of solutions at a penalized level, but it would have been less clear to show how this sequence of penalized solutions converges to a mixed solution of the problem, which is the result we now present.

\begin{Theorem}
There exists at least one mixed solution of $(SOSMFG)$.
\end{Theorem}
\begin{proof}
We begin by introducing the penalized system, using the previous result, we obtain that for all $\epsilon > 0$, there exists a solution $(u_{\epsilon},m_{\epsilon}, \alpha_{\epsilon})$ of the system $(8)$. We now prove uniform estimates and show that the limit is indeed a mixed solution of $(SOSMFG)$. First using $m_{\epsilon}$ as a test function in the equation satisfied by $m_{\epsilon}$ itself we get the bound 

\[
||m_{\epsilon}||_{ H^1_0(\Omega)} \leq C
\]

where $C$ only depends on $\Omega$ and $\rho$. Hence $(f(m_{\epsilon}))_{\epsilon > 0}$ is uniformly bounded in $ L^2(\Omega)$. So we get a uniform bound in $H^2(\Omega)$ for $(u_{\epsilon})_{\epsilon > 0}$. Thus we can find a limit $(u,m)$ in $(H^1_0(\Omega) \cap H^2(\Omega)) \times H^1_0(\Omega)$ for a subsequence of $((u_{\epsilon},m_{\epsilon}))_{\epsilon > 0}$, for the weak topology of $H^2(\Omega) \times H^1_0(\Omega)$. Because of the regularity of $f$ it is clear that $u$ solves the obstacle problem with cost $f(m)$. \\
\\
Since $ - \Delta m_{\epsilon} + m_{\epsilon} \leq \rho$ for all $\epsilon > 0$, we deduce that $ - \Delta m + m \leq \rho$.\\
\\
Taking a smooth function $\phi$ with support in $\{u < 0 \}$, we have for all $\epsilon$ :

\[
 \int_{\Omega} ( - \Delta \phi + \phi) m_{\epsilon} = -  \int_{\Omega} \mathbb{1}_{\{u_{\epsilon} \geq 0\} \cap \{ u < 0 \}}\frac{\alpha_{\epsilon}}{\epsilon} m_{\epsilon} \phi + \rho \phi
\]
Recalling the uniform bound on $(m_{\epsilon})_{\epsilon > 0}$, and the fact that the term under the sum sign in the right hand side converges almost everywhere towards $0$, we apply Fatou's lemma to prove that the right hand side converges toward $0$. Thus passing to the limit in a subsequence which converges to $(u,m)$, we deduce

\[
 \int_{\Omega} ( - \Delta \phi + \phi) m = \int_{\Omega} \rho \phi
\]

Thus,

\[
-\Delta m + m = \rho \text{ in } \{u<0 \}
\]

Now we are going to test the variational formulation of the penalized equation in $u_{\epsilon}$ on $m_{\epsilon}$ and vice versa. Substracting the two equalities yields 

\[
 \int_{\Omega} ( - \Delta u_{\epsilon} + u_{\epsilon}+ \frac{1}{\epsilon} (u_{\epsilon} )^+ - f(m_{\epsilon}))m_{\epsilon} -  \int_{\Omega} ( -\Delta m_{\epsilon} + m_{\epsilon} + \frac{\alpha_{\epsilon}}{\epsilon}\mathbb{1}_{\{u_{\epsilon} \geq 0\}} m_{\epsilon} - \rho)u_{\epsilon} = 0
\]

Hence,

\[
 \int_{\Omega} (\frac{1}{\epsilon} (u_{\epsilon} )^+ - f(m_{\epsilon}))m_{\epsilon} - \frac{\alpha_{\epsilon}}{\epsilon}\mathbb{1}_{\{u_{\epsilon} \geq \}} m_{\epsilon}u_{\epsilon} = -\int_{\Omega} \rho u_{\epsilon}
\]

From which we deduce

\[
\int_{\Omega} f(m_{\epsilon})m_{\epsilon} = \int_{\Omega}u_{\epsilon} \rho
\]

Thanks to the regularity of $f$ we can pass to the limit in the previous equation and we get

\[
\int_{\Omega} f(m)m = \int_{\Omega}u \rho
\]

Let us remark that we can write

\[
\begin{aligned}
\int_{\{ u < 0 \}} f(m)m & = \int_{\{ u < 0 \}}(-\Delta u + u)m \\
& = \int_{\{ u < 0 \}} (- \Delta m + m)u \\
& = \int_{\{u < 0 \}} \rho u \\
& = \int_{\Omega} \rho u
\end{aligned}
\]

which completes the proof of the fact that $(u,m)$ is a mixed solution of $(SOSMFG)$.

\end{proof}

The following result also holds :

\begin{Theorem}
If $f$ is strictly monotone, then there exists at most one solution of $(SOSMFG)$.
\end{Theorem}

The proof of this result is essentially the same as the one we present for the uniqueness of classical solutions in the section 4 so we do not re-write it here.

\part{The time dependent problem}
We present in this part the solution of the obstacle problem linked with the time dependent optimal stopping problem in MFG: $(OSMFG)$. We here define the notion of mixed solutions and present results of existence, and uniqueness under a monotonicity assumption on the costs.

\begin{Def}
A pair $(u,m) \in L^2((0,T),H^1_0(\Omega) \cap H^2(\Omega)) \times L^2((0,T), H^1_0(\Omega))$ is a mixed solution of $(OSMFG)$ if
\begin{itemize}
\item $\max(-\partial_t u - \Delta u - f(m), u - \psi(m)) = 0 $ in $\mathcal{D}'((0,T) \times \Omega)$
\item $\partial_t m - \Delta m \leq 0$  in $\mathcal{D}'((0,T) \times \Omega)$
\item $u = m = 0$ on $\partial \Omega \times (0,T)$ and $u = \psi(m(T))$ at $t=T$
\item $m(0) = m_0$
\item $\partial_t m - \Delta m = 0 $ in $\mathcal{D}'(\{ u < \psi(m) \} )$
\item $\int_{\{u = \psi(m)\}} (f(m) + (\partial_t + \Delta)\psi(m))m = 0$
\end{itemize}
\end{Def}

This definition is the adaptation of the definition $1$ for the case of time dependent problems with an obstacle which depends on $m$. Note that the condition 
\[
\int_{\{u = \psi(m)\}} (f(m) + (\partial_t + \Delta)\psi(m))m = 0
\]

is the analogue of 
 
 \[
 \int_{\{ u = 0 \}} f(m)m = 0
 \]
 
 in the stationary case. Indeed this condition is interpreted as letting possible the fact for $m$ to be strictly positive in the contact region $\{u = \psi(m)\}$, when $u$ satisfies the Hamilton-Jacobi-Bellman equation. Note that we apply the derivatives on $\psi(m)$ which we see as an element of $L^2((0,T), H^2(\Omega))\cap H^1((0,T),L^2(\Omega))$ and not derivatives that we apply on $\psi$ and then evaluate on $m$. 
 
\section{Preliminary results}
We recall in this section some useful results regarding the time dependent obstacle problem, and functions which satisfy what we require the density of players to satisfy.
\subsection{The time dependent obstacle problem}
For any $f_1 \in L^2((0,T), L^2(\Omega))$ and an obstacle $\psi_1 \in L^2((0,T), H^2(\Omega) \cap H^1_0(\Omega)) \cap H^1((0,T), L^2(\Omega))$ there exists a unique $u \in L^2((0,T), H^1_0(\Omega) \cap H^2(\Omega))$ which solves

\[
\max(-\partial_t u -\Delta u - f_1, u- \psi_1) = 0
\]

in the sense of distributions in $(0,T) \times \Omega$ with the terminal condition $u(T) = \psi_1(T)$. This result is the exact analogue of the one we gave in the stationary case. Details about this problem can be found in [24]. Such a $u$ is called the solution of the time dependent obstacle problem with source $f_1$, cost $\psi_1$ and terminal condition $\psi_1(T)$. We recall that the mapping which associates to each pair $(f_1, \psi_1)$ a solution $u$ of the time dependent obstacle problem is continuous from $L^2((0,T), L^2(\Omega)) \times \big(H^1((0,T), L^2(\Omega)) \cap L^2((0,T), H^2(\Omega) \cap H^1_0(\Omega)) \big)$ to $L^2((0,T), H^2(\Omega) \cap H^1_0(\Omega))$. More importantly the sequence of solutions of the penalized system

\[
\begin{cases}
-\partial_t u_{\epsilon} - \Delta u_{\epsilon} + \frac{1}{\epsilon} (u_{\epsilon} - \psi_1)^+ = f_1 \text{ in } (0,T) \times \Omega \\
u_{\epsilon}(T) = \psi_1(T) \\
\forall 0 \leq t \leq T u_{\epsilon}(t) = 0 \text{ on } \partial \Omega
\end{cases}
\]

converges towards the solution of the obstacle problem in $L^2((0,T), H^1_0(\Omega))$.
\subsection{Another useful lemma}
We here introduce a result of comparaison for admissible densities of players. 

\begin{Lemma}
Let $m \in L^2((0,T), H^1_0(\Omega))$ be such that $m(0) = m_0$ and 
\[
\partial_t m - \Delta m \leq 0 \text{ in } \mathcal{D}'((0,T) \times \Omega)
\]

Then for any $v \in L^2((0,T), H^2(\Omega) \cap H^1_0(\Omega)) \cap H^1((0,T), L^2(\Omega))$  such that $v(T) = 0$ and $v \leq 0$,

\[
\int_0^T \int_{\Omega} (-\partial_t v - \Delta v) m - \int_{\Omega} m_0 u(0) \geq 0
\]

with an equality if $\partial_t m - \Delta m =0$ in the interior of the support of $v$.

\end{Lemma}
\begin{proof}
This lemma follows from the definition of the fact that $\partial_t m - \Delta m \leq 0$ in $\Omega$, tested against the function $v$ (which is regular enough).
\end{proof}

\section{Existence of mixed solutions} 
We now turn to the proof of existence of mixed solutions for $(OSMFG)$. As in the stationary case, this result follows from the use of a penalized version of the problem. This is why we introduce the following system for all $\epsilon > 0$ :

\begin{equation}
\begin{cases}
-\partial_t u - \Delta u  + \frac{1}{\epsilon} (u - \psi(m))^+ = f(m) \text{ in} (0,T) \times \Omega \\
\partial_t m - \Delta m  + \frac{1}{\epsilon}\alpha \mathbb{1}_{\{u \geq \psi(m)\}}m = 0 \text{ in } (0,T) \times \Omega\\
 u = m = 0 \text{ on } \partial \Omega \\
 u(T) = \psi(m(T)) \text{ in } \Omega \\
 m(0) = m_0 \text{ in } \Omega \\
0 \leq \alpha \leq 1 ; u \ne \psi(m) \Rightarrow \alpha = 1
\end{cases}
\end{equation}

There exist solutions for this system and the proof of this statement is step by step the same as the one we did in the stationary case, so we do not present it here. The interested reader shall easily be able to adapt all the elliptic arguments into parabolic ones. We prove the following :

\begin{Theorem}
There exists a mixed solution of $(OSMFG)$.
\end{Theorem}
\begin{proof}
We take $(u_{\epsilon}, m_{\epsilon}, \alpha_{\epsilon})_{\epsilon > 0}$ a sequence of solutions of the penalized system. Using the same arguments as in the proof of this statement for the stationary case, we obtain uniform bounds on $(u_{\epsilon}, m_{\epsilon})_{\epsilon > 0}$ and we find a limit $(u,m)$ such that
\begin{itemize}
\item $(u,m) \in L^2((0,T),H^1_0(\Omega) \cap H^2(\Omega)) \times L^2((0,T), H^1_0(\Omega))$
\item $u$ solves the obstacle problem with terminal condition $\psi(m(T))$, source $f(m)$ and obstacle $\psi(m(T))$
\item $\partial_t m - \Delta m = 0$ in $\{ u < \psi(m)\}$ in $\mathcal{D}'(\Omega)$
\end{itemize}
We will here show why $\int_{\{u = \psi(m)\}} (f(m) + (\partial_t + \Delta)\psi(m))m = 0$. Using the equations satisfied by $u_{\epsilon}$ and $m_{\epsilon}$, and multiplying the equation in $u_{\epsilon}$ by $m_{\epsilon}$ and the one in $m_{\epsilon}$ by $u_{\epsilon}$ we can write :

\[
\begin{aligned}
0 = & \int_0^T \int_{\Omega} (-\partial_t u_{\epsilon} - \Delta u_{\epsilon} + \frac{1}{\epsilon} (u_{\epsilon} - \psi(m_{\epsilon}))^+ - f(m_{\epsilon}))m_{\epsilon}\\
& - \int_0^T \int_{\Omega} (\partial_t m_{\epsilon} -\Delta m_{\epsilon} + \frac{\alpha_{\epsilon}}{\epsilon}\mathbb{1}_{\{u_{\epsilon} \geq \psi(m_{\epsilon})\}} m_{\epsilon})u_{\epsilon} 
\end{aligned}
\]

Hence,

\[
\begin{aligned}
 \int_{\Omega} \psi(m_{\epsilon}(T))m_{\epsilon}(T) - u_{\epsilon}(0)m_0 & =  \int_0^T \int_{\Omega} (\frac{1}{\epsilon} (u_{\epsilon} - \psi(m_{\epsilon}))^+ - f(m_{\epsilon}))m_{\epsilon}\\
 & - \int_0^T \int_{\Omega}\frac{\alpha_{\epsilon}}{\epsilon}\mathbb{1}_{\{u_{\epsilon} \geq \psi(m_{\epsilon})\}} m_{\epsilon}u_{\epsilon}
\end{aligned}
\]

We now use the equation verified by $m_{\epsilon}$ to interpret the term $\int_{\Omega} \psi(m_{\epsilon}(T))m_{\epsilon}(T)$

\[
\begin{aligned}
\int_0^T \int_{\Omega} (\frac{1}{\epsilon} (u_{\epsilon} - \psi(m_{\epsilon}))^+ & - f(m_{\epsilon}))m_{\epsilon}  - \frac{\alpha_{\epsilon}}{\epsilon}\mathbb{1}_{\{u_{\epsilon} \geq \psi(m_{\epsilon})\}} m_{\epsilon}u_{\epsilon} = \int_{\Omega} - u_{\epsilon}(0)m_0 \\
& + \int_{\Omega} \psi(m_0)m_0 +\int_0^T \int_{\Omega} ( \partial_t \psi(m_{\epsilon}) + \Delta \psi(m_{\epsilon})\\
& - \int_0^T \int_{\Omega}\frac{\alpha_{\epsilon}}{\epsilon} \mathbb{1}_{\{u_{\epsilon} \geq \psi(m_{\epsilon})\}} \psi(m_{\epsilon}))m_{\epsilon}
\end{aligned}
\]

Using the fact that $(u_{\epsilon} - \psi(m_{\epsilon}))^+ = \alpha_{\epsilon} \mathbb{1}_{\{ u_{\epsilon} \geq \psi(m_{\epsilon}) \}} (u_{\epsilon} - \psi(m_{\epsilon}))$, we deduce

\[
\int_0^T \int_{\Omega} (f(m_{\epsilon}) + \partial_t \psi(m_{\epsilon}) + \Delta \psi(m_{\epsilon}))m_{\epsilon} = \int_{\Omega}  (u_{\epsilon}(0)-\psi(m_0))m_0 
\]

Passing to the limit we obtain

\[
\int_0^T \int_{\Omega} (f(m) + \partial_t \psi(m) + \Delta \psi(m))m = \int_{\Omega}  (u(0)-\psi(m_0))m_0 
\]

Now, because $\partial_t m - \Delta m = 0$ on $\{ u < \psi(m)\}$, we derive finally

\[
\int_{\{u = \psi(m)\}} (f(m) + \partial_t \psi(m) + \Delta \psi(m))m
\]

which completes the proof of the fact that $(u,m)$ is a mixed solution of $(OSMFG)$.
\end{proof}

\section{Uniqueness of mixed solutions}
We now turn to the question of uniqueness of such mixed solutions. We recall that regarding the counter example presented in the first part, we cannot expect uniqueness to hold in general and some assumptions must be made on the two costs $f$ and $\psi$. The previous section strongly suggests to make an assumption on the term $f(m) + (\partial_t + \Delta )\psi(m)$, as it is involved in most of the calculations. One can look as this term using the following remark. The formal obstacle problem

\[
\max(-\partial_t u - \Delta u -f, u - \psi) = 0
\]

can be equivalently reformulated as 

\[
\max(-\partial_t v - \Delta v - \tilde{f},v ) = 0
\]

where $\tilde{f} = f + \partial_t \psi + \Delta \psi$ and $v = u - \psi$, given that $\psi$ is smooth enough. Hence it is natural to adapt the result of uniqueness in the stationary case by making an assumption on the "new cost" $\tilde{f}$.

\begin{Theorem}
Assume that $f +( \partial_t + \Delta )\psi$ is strictly monotone in $L^2((0,T), H^1_0(\Omega))$, then there is a unique solution of the penalized system and a unique mixed solution for $(OSMFG)$.
\end{Theorem}
\begin{proof}
This proof relies on the same arguments as the one we gave for classical solutions of $(SOSMFG)$. We only present here the proof of uniqueness for mixed solutions, since the proof of uniqueness for the penalized system is the same as this one except for the fact that $(u - \psi)^+ = \mathbb{1}_{\{ u \geq \psi\}}(u - \psi)$ plays the role of the integral relation satisfied by the mixed solutions.\\
We denote by $(u_1, m_1)$ and $(u_2, m_2)$ two mixed solutions, $v_i = u_i - \psi(m_i)$, $\Omega_i = \{ v_i < 0 \}$ for $i =1,2$ and $v = v_1 - v_2$ and $m = m_1 - m_2$ and by  $\tilde{f} = f + \partial_t \psi + \Delta \psi$. We can compute

\[
\begin{aligned}
\int_{\Omega_1 \cup \Omega_2} (-\partial_t v - \Delta v)m = & \int_{\Omega_1 \cap \Omega_2} (\tilde{f}(m_1) - \tilde{f}(m_2))m\\
& -\int_{\Omega_1^c \cap \Omega_2} \tilde{f}(m_2)m + \int_{\Omega_2^c \cap \Omega_1} \tilde{f}(m_1)m
\end{aligned}
\]

From which we deduce

\[
\begin{aligned}
\int_{\Omega_1 \cup \Omega_2} (-\partial_t v - \Delta v)m = & \int_{\Omega_1 \cup \Omega_2} (\tilde{f}(m_1) - \tilde{f}(m_2))m\\
& -\int_{\Omega_1^c \cap \Omega_2} \tilde{f}(m_1)m + \int_{\Omega_2^c \cap \Omega_1} \tilde{f}(m_2)m
\end{aligned}
\]

But we know that $(u_1, m_1)$ and $(u_2, m_2)$ are mixed solutions so 

\[
\int_{\Omega_1^c}\tilde{f}(m_1)m_1 = \int_{\Omega_2^c}\tilde{f}(m_2)m_2 = 0
\]

Hence using this relation in the previous equality, we get

\[
\begin{aligned}
\int_{\Omega_1 \cup \Omega_2} (-\partial_t v - \Delta v)m = & \int_{\Omega_1 \cup \Omega_2} (\tilde{f}(m_1) - \tilde{f}(m_2))m\\
& +\int_{\Omega_1^c \cap \Omega_2} \tilde{f}(m_1)m_2 + \int_{\Omega_2^c \cap \Omega_1} \tilde{f}(m_2)m_1
\end{aligned}
\]

Now recalling the argument that because of the obstacles problem $\tilde{f}(m_i) \geq 0 $ on $\Omega_i^c$ and the monotony of $\tilde{f}$ we obtain, as in the stationary case,

\[
\int_{\Omega_1 \cup \Omega_2} (-\partial_t v - \Delta v)m  \geq 0
\]

Next, let us remark that 

\[
\begin{aligned}
\int_0^T \int_{\Omega} (-\partial_t v - \Delta v)m = & \int_0^T \int_{\Omega}(-\partial_t v - \Delta v)m_1 - \int_{\Omega}v(0)m_0\\
& - \big(  \int_0^T \int_{\Omega}(-\partial_t v - \Delta v)m_2 - \int_{\Omega}v(0)m_0  \big)
\end{aligned}
\]

Using lemma $2.1$ we obtain for $i = 1, 2$ and $j \ne i$

\[
\int_0^T \int_{\Omega} (-\partial_t v_i - \Delta v_i)m_i - \int_{\Omega}v_i(0)m_0 = 0
\]
\[
\int_0^T \int_{\Omega} (-\partial_t v_i - \Delta v_i)m_j - \int_{\Omega}v_i(0)m_0 \geq 0
\]

Then, we deduce successively
\[
\int_{\Omega_1 \cup \Omega_2} (-\partial_t v - \Delta v)m \leq 0
\]
\[
\int_{\Omega_1 \cup \Omega_2} (-\partial_t v - \Delta v)m = 0
\]

We can now state as in the stationary case that

\[
\int_{\Omega_1 \cup \Omega_2} (\tilde{f}(m_1) - \tilde{f}(m_2))m = 0
\]

So by strict monotonicity of $\tilde{f}$, $m_1 = m_2$ and there exists a unique solution of the system.

\end{proof}

\section{The optimal control interpretation}
We present here the analogue of the optimal control approach for this more difficult problem. Even if this approach is more restrictive than the approach we just presented, we believe it is useful.\\
\\
Let $\mathcal{H} := \{ m \in L^2((0,T), H^1_0(\Omega)), s.t., m \geq 0, m(0) = m_0, \partial_t m- \Delta m \leq 0 \}$ where the last inequality is taken in the sense of distributions. Suppose there exist $C^1$ potentials $\mathcal{F}$ and $\Psi$ such that 
\[
\frac{\partial \mathcal{F}}{\partial p}(x, m(t,x)) = f(x,m(t,x))
\]

and 
\[
\frac{\partial \Psi}{\partial p}(x, m(t,x)) = (\partial_t + \Delta) \psi(t,m)
\]

Where $p$ stands for the second variable. Then we can prove :

\begin{Theorem}
If $\mathcal{F} + \Psi$ is strictly convex and a continuous functional from $L^2((0,T),H^1_0(\Omega))$ into $L^1((0,T) \times\Omega)$, then there exists a unique minimizer $m$  of 

\[
\inf_{m' \in \mathcal{H}} \int_0^T \int_{\Omega}\mathcal{F}(m') + \Psi(m')
\]

and $(u,m)$ is the only mixed solutions of $(OSMFG)$ if $u$ is set to be the only solution of the time dependent obstacle problem with source $f(m)$, obstacle $\psi(m)$ and terminal cost $\psi(m(T))$.
\end{Theorem}

This theorem relies on the same arguments than the ones we used in the first part. It is natural having in mind the definition of mixed solutions. Indeed in the stationary case the relation satisfied by $m$ in the mixed zone was that  $f(m) = 0$, meaning that the derivative of the function we want to minimize is $0$ when evaluated in $m$ in this zone. Here the relation which is satisfied in the mixed zone is $f(m) + \partial_t \psi(m) + \Delta \psi(m) = 0$, so it is natural to choose a functional such that its derivative with respect to $m$ gives the relation we are looking for. This explanation is of course a heuristic one, but the proof is step by step the same as the one in the stationary case, so we do not present it here.

\section{Remarks on the assumptions on $\psi$}
\subsection{Assumption on the monotonicity}
We made some strong assumptions on $\psi$ in the previous section in order to find uniqueness of solutions. Indeed we assumed that $(\partial_t + \Delta )\psi$ seen as a functional from $L^2((0,T), H^1_0(\Omega))$ to $L^2((0,T), H^2(\Omega) \cap H^1_0(\Omega)) \cap H^1((0,T), L^2(\Omega))$ is strictly monotone. It is not obvious that non-trivial obstacles satisfy such a condition. This is the question we want to address here.\\
\\
We begin with a statement about the model this problem is concerned with. We recall that uniqueness in [19,5] is obtained under the assumption that $f$ is strictly monotone, meaning that for $m'$ bigger than $m$, $f(m')$ is bigger than $f(m)$ and the players pay a higher cost. This assumption tends to force the players to spread. It is then natural to think that under such kind of assumptions, uniqueness may hold. We proved that if $f$ is strictly monotone, then if $\partial_t \psi + \Delta \psi$ is monotone, then uniqueness holds. Remark that the operator $\partial_t + \Delta$ (which is defined only when boundary conditions are imposed, but this is not the purpose of this formal discussion) is decreasing in the following sense if $\partial_t f_1 + \Delta f_1$ is bigger that  $\partial_t f_2 + \Delta f_2$ then $f_2$ is bigger than $f_1$. This means that making the assumption that $\partial_t \psi + \Delta \psi$ is monotone implies that if $m_1$ is bigger than $m_2$, then $\psi(m_2)$ is bigger than $\psi(m_1)$. Note this can be thought as an incitation to leave the game when the players are more. This can be interpreted as in [19,5] as a tendency for the players not to accumulate. From this point of vue the assumption we make in order to have uniqueness is natural. Moreover, it is easy to show that under the assumption that $\partial_t \psi + \Delta \psi$ is monotone, $\psi$ is anti-monotone.\\
\\
We are going to show here that there are non-trivial $\psi$ such that for all $m_1$ and $m_2$ in $L^2((0,T), H^1_0(\Omega))$ the following inequality holds true

\begin{equation}
\int_0^T \int_{\Omega} (\partial_t( \psi(m_1) - \psi(m_2))+ \Delta (\psi(m_1) - \psi(m_2)) ) (m_1 - m_2) \geq 0
\end{equation}

Now let us remark that any $\psi$ which does not depend on $m$ satisfies such an inequality. This makes this assumption coherent with the results of the first part. We here denote by $g$ a monotone operator and define $\psi$ by : given $m$ in $L^2((0,T), H^1_0(\Omega))$, $\psi(m)$ is the only solution of

\[
\begin{cases}
\partial_t \psi(m) + \Delta \psi(m) = g(m) \text{ in } (0,T) \times \Omega \\
\psi(m)(T) = 0 \text{ in } \Omega \\
\forall 0 \leq t \leq T, \psi(m)(t) = 0 \text{ on } \partial \Omega
\end{cases}
\]

Clearly the application $\psi$ defined with this equation is regular enough to be taken as an obstacle in $(OSMFG)$ and satisfied (9).

\subsection{Assumption on the existence of a primitive}
We make here precise the sense in which we can understand the assumption we made in the optimal control approach. We assume there exists a $\Psi$ such that the derivative of $\Psi$ with respect to $m$ is $\partial_t \psi + \Delta \psi$. Of course derivatives in the space of measure are now well understood, see for example the section 2 in [7], however we are not interested in presenting this formalism here and we want to introduce a primitive in $m$ in a local sense, just like we did for the cost $f$. That is why we want $\partial_t \psi + \Delta \psi$ to have a local dependence in $m$, meaning that 

\[
(\partial_t \psi(m) + \Delta \psi(m))(t,x) = g(t,x, m(t,x))
\]

In which $g$ is a smooth function in the second variable. $\Psi$ is then simply the primitive of $g$ with respect to its second argument. Note that there are $\psi$ such that such a $g$ exists, just define for example $\psi(m)$ implicitly as the only solution of 

\[
\begin{cases}
\partial_t \psi(m) + \Delta \psi(m) = g(m) \text{ in }(0,T) \times \Omega \\
\psi(m) = 0 \text{ on } \big( \{T\} \times \Omega \big) \cup \big((0,T) \times \partial \Omega \big)
\end{cases}
\]
Hence there exist non trivial $\psi$ such that we can find a potential $\Psi$ for which 

\[
\frac{\partial \Psi}{\partial p } (t,x,m(t,x)) = (\partial_t + \Delta)(\psi(m))(t,x)
\]

\part{The case of optimal stopping with continuous control}
In this last part, we extend the case of optimal stopping to a case in which both the classical optimal control and optimal stopping can occur, in a MFG setting. We work in the case in which $\psi$ is equal to $0$ because some technical difficulties arise for general obstacles and we do not want to enter in those details. Mainly we are interested with existence and uniqueness of mixed solutions of the following forward-backward system :

\begin{equation*}{(COSMFG)}
\begin{cases}
\max(-\partial_t u - \Delta u + H(x, \nabla u) - f(m), u) = 0 \text{ in } (0,T) \times \Omega \\
\partial_t m - \Delta m - div(m D_p H(x, \nabla u)) = 0 \text{ in } \{ u < 0 \} \\
m = 0 \text{ in } \{ u = 0 \} \\
m(0) = m_0 \text{ and } u(T) = 0 \text{ in } \Omega
\end{cases}
\end{equation*}

We call mixed solution of this problem a couple $(u,m) \in L^2((0,T), H^2(\Omega) \cap H^1_0(\Omega)) \times L^2((0,T), H^1_0(\Omega))$ such that
\begin{itemize}
\item $\max(-\partial_t u - \Delta u + H(x, \nabla u) - f(m), u ) = 0 \text{ in } \mathcal{D}'((0,T) \times \Omega)$
\item $\partial_t m - \Delta m - div(m D_p H(x, \nabla u)) = 0 \text{ in } \mathcal{D}'(\{ u < 0 \})$
\item $\partial_t m - \Delta m - div(m D_p H(x, \nabla u)) \leq 0 \text{ in } \mathcal{D}'((0,T) \times \Omega)$
\item $\int_{\{ u = 0 \}} (f(m) - H(x, 0))m = 0 $
\item $m(0) = m_0 \text{ and } u(T) = 0 \text{ in } \Omega$
\end{itemize}

We keep on with the assumptions we made on the regularity of  $f$ and we precise here assumptions on the hamiltonian $H$ :

\begin{itemize}
\item $H$ is lipschitz in both variables and convex in the second variable
\item $\exists C > 0, \forall x, p \in \Omega \times \mathbb{R}^d, | H(x, p) | \leq C |p|^{2^*}$
\end{itemize}

where $2^{*} = \frac{2d}{d- 2}$ when $d \geq 3$ and any number greater than one when $d \leq 2$.

As in the previous part, we state an analogue of lemma $1.1$ before presenting the results on existence and uniqueness of solutions of this problem. We also present the optimal control interpretation of this problem. We can establish the following result :

\begin{Lemma}
Let $m \in L^2((0,T), H^1_0(\Omega))$, $g \in L^{\frac{2^*}{2^* - 1}}((0,T)\times \Omega)$ be such that $m(0) = m_0$ and 
\[
\partial_t m - \Delta m - div(m g) \leq 0 \text{ in } \mathcal{D}'(\Omega)
\]
Then for any $\phi \in L^2((0,T), H^2(\Omega) \cap H^1_0(\Omega)) \cap H^1((0,T), L^2(\Omega))$ such that $\phi \leq 0$ and $\phi(T) = 0$,

\[
\int_0^T \int_{\Omega} (-\partial_t \phi - \Delta \phi + g\cdot \nabla \phi) m \geq \int_{\Omega} \phi(0)m_0
\]

with an equality if $\partial_t m - \Delta m - div(mg) = 0$ on the support of $\phi$.

\end{Lemma}
\begin{proof}
This statement is only the variational interpretation of the inequality 
\[
\partial_t m - \Delta m - div(m g) \leq 0
\]
Taking into account the boundary conditions.
\end{proof}

\section{Existence of mixed solutions}
Once again we prove that there exist mixed solutions for $(COSMFG)$ using a penalized version of it. Since the proof is again a very easy adaptation of the proof of existence for $(SOSMFG)$, we do not detail some arguments. We prove the following result

\begin{Theorem}
There exists at least one mixed solution of $(COSMFG)$.
\end{Theorem}
\begin{proof}
As we did in the previous two parts we are going to introduce the following penalized version of $(COSMFG)$ 

\[
\begin{cases}
-\partial_t u - \Delta u + H(x, \nabla u) + \frac{1}{\epsilon}(u)^+ = f(m) \text{ in } (0,T) \times \Omega \\
\partial_t m - \Delta m - div(m D_{p}H(x, \nabla u)) + \frac{\alpha}{\epsilon} \mathbb{1}_{\{ u \geq 0 \} } m = 0 \text{ in } (0,T) \times \Omega \\
m(0) = m_0 \text{ and } u(T) = 0 \text{ in } \Omega \\
m = u = 0 \text{ on } \partial \Omega \\
u \ne 0 \Rightarrow \alpha = 1
\end{cases}
\]

The proof of existence of solutions for such a system is the same as the one we gave in the case of $(SOSMFG)$, so we do not present it here. We have a sequence $((u_{\epsilon}, m_{\epsilon}))_{\epsilon > 0}$ of elements of $(L^2((0,T), H^2(\Omega) \cap H^1_0(\Omega)))^2$ of solutions of the penalized system. Using a priori estimates on the equations ([21]) we deduce that, up to a subsequence, $((u_{\epsilon}, m_{\epsilon}))_{\epsilon > 0}$ converges to $(u,m)$ for the topology of $L^2((0,T), H^1(\Omega)) \times L^2((0,T), L^2(\Omega))$  with $(u,m) \in L^2((0,T), H^2(\Omega) \cap H^1_0(\Omega)) \times L^2((0,T), H^1_0(\Omega))$. Because of the regularity of $f$ and $H$, $u$ solves 
\[
\max(-\partial_t u - \Delta u + H(x, \nabla u) - f(m), u ) = 0 \text{ in } (0,T) \times \Omega
\]

and we obtain for $m$ that

\[
\begin{cases}
\partial_t m - \Delta m - div(m D_p H(x, \nabla u)) \leq 0 \text{ in } (0,T) \times \Omega \\
\partial_t m - \Delta m - div(m D_{p} H(x, \nabla u)) = 0 \text{ in } \{ u < 0 \}
\end{cases}
\]

We now check that  $\int_{\{ u = 0 \}} (f(m)  - H(x, 0))m = 0$ to prove that $(u,m)$ is indeed a mixed solution of $(COSMFG)$. As in the previous parts we multiply the equation in $u$ by $m$ and we integrate, and vice-versa for the equation in $m$ (at the penalized level). This leads to 

\[
\begin{aligned}
 0 = & \int_0^T \int_{\Omega} (-\partial_t u_{\epsilon} - \Delta u_{\epsilon} + H(x, \nabla u_{\epsilon}) + \frac{1}{\epsilon}(u_{\epsilon})^+ - f(m_{\epsilon}))m_{\epsilon} \\
 & - \int_0^T \int_{\Omega}(\partial_t m_{\epsilon} - \Delta m_{\epsilon} - div(m_{\epsilon} D_{p}H(x, \nabla u_{\epsilon})) + \frac{\alpha}{\epsilon} \mathbb{1}_{\{ u_{\epsilon} \geq 0\} } m_{\epsilon})(u_{\epsilon})
\end{aligned}
\]

By integration by parts in time, we obtain, using the fact that $(u)^+ = \mathbb{1}_{\{ u \geq 0 \} }u$
\[
\begin{aligned}
0 = & \int_{\Omega} (u_{\epsilon}(0) )m_0 + \int_0^T \int_{\Omega}( H(x, \nabla u_{\epsilon}) - f(m_{\epsilon})) m_{\epsilon}\\
& + \int_0^T \int_{\Omega} div(m_{\epsilon} D_{p}H(x, \nabla u_{\epsilon}))(u_{\epsilon})  
\end{aligned}
\]

because of the convergence (up to a subsequence) of $((u_{\epsilon}, m_{\epsilon}))$ we can pass to the limit in this equality and we get

\[
\begin{aligned}
0 = & \int_{\Omega} (u(0))m_0 + \int_0^T \int_{\Omega}( H(x, \nabla u) - f(m)) m\\
& + \int_0^T \int_{\Omega} div(m D_{p}H(x, \nabla u))(u)
\end{aligned} 
\]

Then using lemma $3.1$ we deduce that

\[
0 =   \int_0^T \int_{\Omega}( -\partial_t u - \Delta u + H(x, \nabla u) - f(m)) m 
\]

Which leads to

\[
\int_{\{ u = 0 \}} (f(m) - H(x, 0))m = 0
\]

\end{proof}

\section{Uniqueness of mixed solutions}
We here show that adding a convex hamiltonian do not alter the property of uniqueness of the problem.
\begin{Theorem}
If $f$ is strictly monotone, then there exists a unique mixed solution of $(COSMFG)$.
\end{Theorem}
\begin{proof}
The proof here is once again the adaptation of the one we gave in the first part. We take $(u_1, m_1)$ and $(u_2, m_2)$ two mixed solutions of $(COSMFG)$ and we are going to show that 

\[
\int_0^T \int_{\Omega} ((-\partial_t - \Delta)(u_1 - u_2) + H(x, \nabla u_1) - H(x, \nabla u_2))(m_1 - m_2) = 0
\]
 First we note the sets $\Omega_i = \{ u_i < 0 \}$ for $i = 1, 2$ and we compute
 
 \[
 \begin{aligned}
 & \int_0^T \int_{\Omega} ((-\partial_t - \Delta)(u_1 - u_2) + H(x, \nabla u_1) - H(x, \nabla u_2))(m_1 - m_2) =\\
 & \int_{\Omega_1 \cap \Omega_2} (f(m_1) - f(m_2))(m_1 - m_2) + \int_{\Omega_1 \cap \Omega_2^c} (f(m_1) - H(x, 0))(m_1 - m_2) \\
 & + \int_{\Omega_2 \cap \Omega_1^c} (f(m_2) - H(x, 0))(m_2 - m_1) 
 \end{aligned}
 \]
 
Once again, using the same arguments we used in the previous part we deduce 
\[
\begin{aligned}
&\int_0^T \int_{\Omega} ((-\partial_t - \Delta)(u_1 - u_2) + H(x, \nabla u_1) - H(x, \nabla u_2))(m_1 - m_2) \geq \\
& \int_0^T \int_{\Omega}(f(m_1) - f(m_2))(m_1 - m_2)
\end{aligned}
\]

Now we obtain using lemma $3.1$ that
 
 \[
 \begin{aligned}
 \int_0^T \int_{\Omega} ((-\partial_t - \Delta)(u_1 - u_2))(m_1 - m_2) \leq  & \int_0^T \int_{\Omega} (D_{p} H(x, \nabla u_1) \cdot \nabla (u_2 - u_1) )m_1\\
 & + \int_0^T \int_{\Omega} (D_{p} H(x, \nabla u_2) \cdot \nabla (u_1 - u_2) )m_2
 \end{aligned}
 \]
 
 Hence, using the convexity of $H$ we get that
 
 \[
 \int_0^T \int_{\Omega} ((-\partial_t - \Delta)(u_1 - u_2))(m_1 - m_2) \leq \int_0^T \int_{\Omega} (H(x, \nabla u_2) - H(x, \nabla u_1))(m_1 - m_2)
 \]
 
 So we have proven
 \[
\int_0^T \int_{\Omega} ((-\partial_t - \Delta)(u_1 - u_2) + H(x, \nabla u_1) - H(x, \nabla u_2))(m_1 - m_2) = 0
\]

and we conclude as we did before.
\end{proof}

\section{The optimal control interpretation}
We end this part with the presentation of the optimal control interpretation of $(COSMFG)$. We are not interested in giving results or proofs in this part as it is only the mix of what we can found in [6,8] and what we did in the optimal control interpretation of $(SOSMFG)$. We just want to show to the reader to which optimization problem $(COSMFG)$ is linked. Moreover we do not want to add some technical difficulties here. We suppose that there exists $\mathcal{F} \in C^1(\Omega \times \mathbb{R})$ such that

\[
\forall (x,p) \in \Omega \times \mathbb{R}, \frac{\partial \mathcal{F}}{\partial p}(x, p) = f(x,p)
\]

The following optimization problem is then naturally linked to $(COSMFG)$

\[
\inf_{(\alpha, m) \in \mathcal{H}} \big\{ \int_0^T \int_{\Omega} \mathcal{F}(m) - H(x, 0)m +  L(x, \alpha)m \big\}
\]

where $\mathcal{H}$ is the set of $(\alpha, m)$ in $L^2((0,T) , H^1(\Omega)) \times L^2((0,T), H^2(\Omega)\cap H^1_0(\Omega))$, such that

\[
\partial_t m - \Delta m -div(\alpha m) \leq 0
\]

in the sense of distributions. $L$ is here the Fenchel conjugate of $H$ with respect to the second variable. Such a problem is convex and under strict convexity of $\mathcal{F}$, we have existence of a unique minimizer.

\section*{Acknowledgments} 
I would like to thank Pierre-Louis Lions (Collège de France) for his helpful advices. 

\nocite{*}
\bibliographystyle{plain}
\bibliography{osmfg}

\begin{thebibliography}{10}

\bibitem{achdou2010mean}
Yves Achdou and Italo Capuzzo-Dolcetta.
\newblock Mean field games: Numerical methods.
\newblock {\em SIAM Journal on Numerical Analysis}, 48(3):1136--1162, 2010.

\bibitem{bardi2008optimal}
Martino Bardi and Italo Capuzzo-Dolcetta.
\newblock {\em Optimal control and viscosity solutions of
  Hamilton-Jacobi-Bellman equations}.
\newblock Springer Science \&amp; Business Media, 2008.

\bibitem{benamou2015augmented}
Jean-David Benamou and Guillaume Carlier.
\newblock Augmented lagrangian methods for transport optimization, mean field
  games and degenerate elliptic equations.
\newblock {\em Journal of Optimization Theory and Applications}, 167(1):1--26,
  2015.

\bibitem{caffarelli1998obstacle}
Luis~A Caffarelli.
\newblock The obstacle problem revisited.
\newblock {\em Journal of Fourier Analysis and Applications}, 4(4):383--402,
  1998.

\bibitem{cardaliaguet2010notes}
Pierre Cardaliaguet.
\newblock Notes on mean field games.
\newblock Technical report, Technical report, 2010.

\bibitem{cardaliaguet2015weak}
Pierre Cardaliaguet.
\newblock Weak solutions for first order mean field games with local coupling.
\newblock In {\em Analysis and Geometry in Control Theory and its
  Applications}, pages 111--158. Springer, 2015.

\bibitem{cardaliaguet2015master}
Pierre Cardaliaguet, Fran{\c{c}}ois Delarue, Jean-Michel Lasry, and
  Pierre-Louis Lions.
\newblock The master equation and the convergence problem in mean field games.
\newblock {\em arXiv preprint arXiv:1509.02505}, 2015.

\bibitem{cardaliaguet2015second}
Pierre Cardaliaguet, P~Jameson Graber, Alessio Porretta, and Daniela Tonon.
\newblock Second order mean field games with degenerate diffusion and local
  coupling.
\newblock {\em Nonlinear Differential Equations and Applications NoDEA},
  22(5):1287--1317, 2015.

\bibitem{cardaliaguet2015learning}
Pierre Cardaliaguet and Saeed Hadikhanloo.
\newblock Learning in mean field games: the fictitious play.
\newblock {\em arXiv preprint arXiv:1507.06280}, 2015.

\bibitem{cardaliaguet2012long}
Pierre Cardaliaguet, Jean-Michel Lasry, Pierre-Louis Lions, Alessio Porretta,
  et~al.
\newblock Long time average of mean field games.
\newblock {\em NHM}, 7(2):279--301, 2012.

\bibitem{carmona2013probabilistic}
Ren{\'e} Carmona and Fran{\c{c}}ois Delarue.
\newblock Probabilistic analysis of mean-field games.
\newblock {\em SIAM Journal on Control and Optimization}, 51(4):2705--2734,
  2013.

\bibitem{carmona2016mean}
Rene Carmona, Fran{\c{c}}ois Delarue, and Daniel Lacker.
\newblock Mean field games of timing and models for bank runs.
\newblock {\em arXiv preprint arXiv:1606.03709}, 2016.

\bibitem{gomes2014obstacle}
Diogo Gomes and Stefania Patrizi.
\newblock Obstacle mean-field game problem.
\newblock {\em Interfaces and Free Boundaries}, 17(1):55--68, 2015.

\bibitem{gueant2009reference}
Olivier Gu{\'e}ant.
\newblock A reference case for mean field games models.
\newblock {\em Journal de math{\'e}matiques pures et appliqu{\'e}es},
  92(3):276--294, 2009.

\bibitem{gueant2011mean}
Olivier Gu{\'e}ant, Jean-Michel Lasry, and Pierre-Louis Lions.
\newblock Mean field games and applications.
\newblock In {\em Paris-Princeton lectures on mathematical finance 2010}, pages
  205--266. Springer, 2011.

\bibitem{lacker2015mean}
Daniel Lacker.
\newblock Mean field games via controlled martingale problems: existence of
  markovian equilibria.
\newblock {\em Stochastic Processes and their Applications}, 125(7):2856--2894,
  2015.

\bibitem{lasry2006jeux}
Jean-Michel Lasry and Pierre-Louis Lions.
\newblock Jeux {\`a} champ moyen. i--le cas stationnaire.
\newblock {\em Comptes Rendus Math{\'e}matique}, 343(9):619--625, 2006.

\bibitem{lasry2006jeux2}
Jean-Michel Lasry and Pierre-Louis Lions.
\newblock Jeux {\`a} champ moyen. ii--horizon fini et contr{\^o}le optimal.
\newblock {\em Comptes Rendus Math{\'e}matique}, 343(10):679--684, 2006.

\bibitem{lasry2007mean}
Jean-Michel Lasry and Pierre-Louis Lions.
\newblock Mean field games.
\newblock {\em Japanese Journal of Mathematics}, 2(1):229--260, 2007.

\bibitem{lewy1969regularity}
Hans Lewy and Guido Stampacchia.
\newblock On the regularity of the solution of a variational inequality.
\newblock {\em Communications on Pure and Applied Mathematics}, 22(2):153--188,
  1969.

\bibitem{lions1983hamilton}
Pierre-Louis Lions.
\newblock Hamilton-jacobi-bellman equations and the optimal control of
  stochastic systems.
\newblock In {\em Proceedings of the International Congress of Mathematicians,
  Warsaw}, 1983.

\bibitem{lions2007cours}
Pierre-Louis Lions.
\newblock Cours au college de france.
\newblock {\em www.college-de-france.fr}, 2011, 2007.

\bibitem{nutz2016mean}
Marcel Nutz.
\newblock A mean field game of optimal stopping.
\newblock {\em arXiv preprint arXiv:1605.09112}, 2016.

\bibitem{touzi2002stochastic}
Nizar Touzi.
\newblock Stochastic control and application to finance.
\newblock {\em Scuola Normale Superiore, Pisa. Special Research Semester on
  Financial Mathematics}, 2002.

\end{thebibliography}
\end{document}